\documentclass[preprint,12pt]{elsarticle}
\usepackage{indentfirst,enumerate,amssymb,amsfonts,amsmath,amsthm,mathrsfs}
\usepackage{color}
\usepackage[colorlinks=true,linkcolor=blue,citecolor=blue]{hyperref}

\newtheorem{theorem}{Theorem}[section]
\newtheorem{lemma}[theorem]{Lemma}
\newtheorem{example}[theorem]{Example}
\newtheorem{proposition}[theorem]{Proposition}
\newtheorem{remark}[theorem]{Remark}

\newtheorem{definition}[theorem]{Definition}
\numberwithin{equation}{section}

\begin{document}
\begin{frontmatter}
\title{Global hypoellipticity for strongly invariant operators}

%----------Author 1

\author{Alexandre Kirilov}
\address{
	Universidade Federal do Paran\'{a}, 
	Departamento de Matem\'{a}tica, \\
	C.P.19096, CEP 81531-990, Curitiba, Brazil
}
\ead{akirilov@ufpr.br}

%----------Author 2

\author{Wagner A. A. de Moraes}
\address{
	Universidade Federal do Paran\'{a},
	Programa de P\'os-Gradua\c c\~ao em Matem\'{a}tica,\\
	C.P.19096, CEP 81531-990, Curitiba, Brazil
}
\ead{wagneramat@gmail.com}

%---------Author 3

%----------classification, keywords, date
%\thanks{This study was financed in part by the Coordenação de Aperfeiçoamento de Pessoal de Nível Superior - Brasil (CAPES) - Finance Code 001.}

\date{\today}

\begin{abstract}
In this note, by analyzing the behavior at infinity of the matrix symbol of an invariant operator $P$ with respect to a fixed elliptic operator, we obtain a necessary and sufficient condition to guarantee that $P$ is globally hypoelliptic. As an application, we obtain the characterization of global hypoellipticity on compact Lie groups and examples on the sphere and the torus. We also investigate relations between the global hypoellipticity of $P$ and global subelliptic estimates. 
\end{abstract}
\end{frontmatter}

% -------------------- for long articles
%\tableofcontents 
% --------------------

%========================================================================================================
\section{Introduction}
%========================================================================================================

This note aims to study the global hypoellipticity of strongly invariant operators defined on a closed smooth manifold $M$. More precisely, consider a linear continuous operator  $P:{\mathscr{D}'}(M) \rightarrow {\mathscr{D}'}(M)$ that commutes with an elliptic operator $E$ defined on $M$ and assume that the domain of the adjoint operator $P^*$ contains $C^\infty(M)$.

The assumption of commutativity introduces on $M$ a Fourier analysis relative to the elliptical operator $E$ and the assumption on the domain of the adjoint operator ensures that the Fourier coefficients of $Pu$ are the product of its matrix symbol $\sigma_P$ by the Fourier coefficient of $u\in C^\infty(M)$. For more details, see Section 4 of \cite{DR-Fmul}.

We recall that an operator $L$ is globally hypoelliptic on $M$ if the conditions $u \in {\mathscr{D}'}(M)$ and $L u \in C^{\infty}(M)$ imply $u \in C^{\infty}(M)$. This global property has been widely studied on the torus, see  \cite{BERG94,BCM93,BDG17,BDGK15,CC00,GPY1,GPY2,GW72-Proc,GW73-TAMS,HP00,HP02,Hou79,HOU82,Petr06,Petr11}, and on compact Lie groups, see \cite{RTW14,DR14,RW15}.

The first study on the global hypoellipticity of differential operators that commute with an elliptic operator on a closed manifold was presented by S. Greenfield and N. Wallach in 1973, see \cite{GW73-TAMS}. More recently, in \cite{DR-Fmul}, J. Delgado and M. Ruzhansky have developed a theory on strongly invariant operators by obtaining a precise characterization of the necessary and sufficient conditions to construct a consistent Fourier analysis with respect to an elliptic operator on a closed manifold. Using this characterization, in \cite{AK19,AGK18} was studied the global hypoellipticity in a class of strongly invariant operators with separation of variables in a specific Cartesian product of compact manifolds. 

In this note, we use the characterization obtained by Delgado and Ruzhansky to characterize the global hypoellipticity and to extend the results obtained by Greenfield and Wallach to the context of strongly invariant operators defined on a closed manifold.

First, in Section \ref{sect-notation}, we introduce the notation and the results necessary for the development of this note. Next, in Section \ref{sect-GW}, we present a version, for strongly invariant operators, of Greenfield's and Wallach's classical theorem, which relates global hypoellipticity of an operator to the behavior of its symbol at infinity. As an application, in Section \ref{Liegroup_section}, we introduce the notation necessary to translate our main result into the context of Lie groups, and we present concrete examples of globally hypoelliptic operators on the sphere $\mathbb{S}^3$ and the torus $\mathbb{T}^2$. Finally, in Section \ref{sect-subelliptic}, we study some of the connections between global hypoellipticity and the validity of global subelliptic estimates.

%========================================================================================================
\section{Fourier analysis associated to an elliptic operator}\label{sect-notation}
%========================================================================================================

Let $\mathbb{N}_0=\mathbb{N}\cup\{0\}$, $\langle \cdot, \cdot \rangle$ be the usual inner product of $\mathbb{C}^d$ and $M$ be a $d$--dimensional closed smooth manifold endowed with a positive measure $dx$. Consider the space $L^2(M)$ of square integrable complex-valued functions on $M$ with respect to $dx$ and denote by $H^s(M)$ the standard Sobolev space of order $s$ on $M$, thus
\begin{equation*}\label{sobolev-dx}
C^\infty(M) = \bigcap_{s \in \mathbb{R}}H^{s}(M) \mbox{ \ and \ } {\mathscr{D}'}(M) = \bigcup_{s \in \mathbb{R}}H^{s}(M).
\end{equation*}

Following the construction proposed by J. Delgado and M. Ruzhansky (see \cite{DR-Fmul}), we introduce a discrete Fourier analysis in $M$ that is associated to an elliptic operator. Let $E=E(x,D_x)$ be a fixed classical positive elliptic pseudo-differential operator of order $\nu\in \mathbb{R}$, then:
 \begin{enumerate}
	\item the eigenvalues of E, counted without multiplicities, form a sequence
	\begin{equation}\label{lambda>0}
	0=\lambda_0 < \lambda_1 < \lambda_2 < \ldots  \longrightarrow \infty;
	\end{equation}
	\item for each $\lambda_j$, the eigenspace $E_{\lambda_j}$ of $E$ has finite dimension $d_j$, $E_{\lambda_j}$ is a subspace of $C^\infty(M)$  and
	\begin{equation}\label{serie}
	\sum_{j=0}^{\infty}d_j(1+\lambda_j)^{-2n} < \infty.
	\end{equation}
	
	\item there is an orthonormal basis $\{e_{j}^{k}\ ; {1 \leq k \leq d_j} \mbox{ and } {j \in \mathbb{N}_0} \}$ for $L^2(M)$ consisting of smooth eigenfunctions of $E$ such that for each $j\in \mathbb{N}_0$, $\{e_{j}^{1},e_{j}^{2},\ldots, e_{j}^{d_j}\}$ is an orthonormal basis of $E_{\lambda_j}$ and  $$L^2(M) = \bigoplus_{j\in \mathbb{N}_0} \, E_{\lambda_j};$$
	
	\item the Fourier coefficients of $f \in L^2(M)$, with respect to this orthonormal basis, are given by
	$$ \widehat{f}(j,k)\doteq \int_{M}{f(x)\overline{e_{j}^{k}(x)}dx}, \ 1\leq k \leq d_j, \  j \in \mathbb{N}_0.$$
	We also write $\widehat{f}(j) = \big( \widehat{f}(j,1), \ldots, \widehat{f}(j,d_j)\big), \ j \in \mathbb{N}_0$;
	
	\medskip
	\item if $u \in {\mathscr{D}'}(M)$, then $\widehat{u}(j,k) \doteq u (\overline{e_j^k})$ and
	\begin{equation*}\label{Fourier1}
	u = \sum_{j \in \mathbb{N}_0} \sum_{k=1}^{d_j} \widehat{u}(j,k) e_j^k(x) = \sum_{j\in \mathbb{N}_0}\left\langle  \widehat{u}(j), \overline{e_j}(x) \right\rangle,
	\end{equation*}
	where $\widehat{u}(j) = \big( \widehat{u}(j,1), \ldots, \widehat{u}(j,d_j)\big)$ and  $e_j(x) = \big(e_j^1(x),\ldots,e_j^{d_j}(x) \big)$;
	
	\medskip
	\item smooth functions on $M$ are characterized by
	{\small \begin{equation}\label{f-smooth}
	f \in C^{\infty}(M) \Leftrightarrow \forall N \in \mathbb{N}, \exists C_N>0, \forall \ell \in \mathbb{N}, \ \|\widehat{f}(\ell)\| \leq C_N(1+\lambda_\ell)^{-N};
	\end{equation}}
	and, by duality, distributions are characterized by
	\begin{equation}\label{u-distrib}
	u \in {\mathscr{D}'}(M) \Leftrightarrow \exists N \in \mathbb{N}, \exists C>0, \forall \ell \in \mathbb{N}, \ \|\widehat{u}(\ell)\| \leq Cf(1+\lambda_\ell)^N.
	\end{equation}
	
	\item for a distribution $u \in {\mathscr{D}'}(M)$ we have
	$$u \in H^s(M) \Leftrightarrow  \sum\limits_{j=0}^{\infty}\sum\limits_{k=1}^{d_j}(1+\lambda_j)^{2s/\nu}|\widehat{u}(j,k)|^2 < \infty.$$
\end{enumerate}

The next results and definitions are a consequence of the results and remarks in Section 4 of \cite{DR-Fmul}.

\begin{proposition}\label{DR_str_inv} 
	Let $P: C^{\infty}(M) \rightarrow C^{\infty}(M)$ be a linear operator. If the domain of $P^*$ contains $C^{\infty}(M)$, then the following conditions are equivalent: 
	\begin{enumerate}[(i)]
		\item For each $j\in\mathbb{N}_0$, we have $P(E_{\lambda_j})\subset E_{\lambda_j}$. 
		\item For each $j\in\mathbb{N}_{0}$ and $1\leq k\leq j$, we have
		$PE e_{j}^{k}=EP e_{j}^{k}.$
		\item For each $\ell\in\mathbb{N}_0$ there exists a matrix $\sigma(\ell)\in\mathbb{C}^{d_{\ell}\times d_{\ell}}$ such that for all $e_j^k$ 
		\begin{equation}\label{invadef}
		\widehat{Pe_j^k}(\ell,m)=\sigma(\ell)_{mk}\delta_{j\ell}.
		\end{equation}
		\item For each $\ell\in\mathbb{N}_0 $ there exists a matrix $\sigma(\ell)\in\mathbb{C}^{d_{\ell}\times d_{\ell}}$ such that
		\begin{equation}\label{invadef2}
		\widehat{Pf}(\ell)=\sigma(\ell)\widehat{f}(\ell), \ f\in C^{\infty}(M).
		\end{equation}
	\end{enumerate}
	
	The matrices $\sigma(\ell)$ in \eqref{invadef} and in \eqref{invadef2} coincide.	Moreover, if $P$ extends to a linear continuous operator $P: {\mathscr{D}'}(M) \rightarrow {\mathscr{D}'}(M)$, then the above properties are also equivalent to: %the following ones: \vspace{-3mm}
	\begin{enumerate}[(i)]
		\setcounter{enumi}{4}
		%\item For each $j \in \mathbb{N}_0$, we have $PP_j = P_jP$ on $ C^{\infty}(M)$.
		\item $PE=EP$ on $L^2(M)$.
	\end{enumerate}
\end{proposition}

\begin{definition}
If any of the equivalent conditions $(i)-(iv)$ are satisfied, we say that the operator $P$ is invariant with respect to $E$ (or simply $E$-invariant) and its matrix symbol is the sequence $\sigma_P$ of matrices given by properties $(iii)$ and $(iv)$. 

If $P$ extends to a linear continuous operator $P: {\mathscr{D}'}(M) \rightarrow {\mathscr{D}'}(M)$ and satisfies any of the equivalent conditions $(i)-(v)$, we say that $P$ is strongly invariant with respect to $E$. 
\end{definition}

Any $E$-invariant operator $P$ can be written in the following way:
\begin{equation}\label{45}
Pf(x) =\sum_{\ell=0}^{\infty} \sum_{m=1}^{d_\ell}{(\sigma_P(\ell)\widehat{f}(\ell))_me_\ell^m(x)} = \sum_{\ell=0}^{\infty}\left[\sigma_P(\ell)\widehat{f}(\ell)\right]e_\ell(x),
\end{equation}
In particular,
\begin{equation}
\label{46}
Pe_j^k(x)=\sum_{m=1}^{d_j}{\sigma_P(j)_{mk}e_j^m(x)}.
\end{equation}

\begin{proposition}\label{sobolev-cont}
	Let $P$ be an $E$-invariant operator with symbol $\sigma_{P}$ satisfying the following property: there exist $C>0$  and $m\in\mathbb R$ such that 
	$$
	\|\sigma_P(\ell)\|_{\mathscr{L}(E_{\lambda_\ell})} \leq C(1+\lambda_\ell)^{m/\nu}, \ \ell\in\mathbb{N}_{0},
	$$
	where $\|\sigma_P(\ell)\|_{\mathscr{L}(E_{\lambda_\ell})}$ denotes the operator norm in $E_{\lambda_\ell}$.
	Then, $P$ extends to a bounded operator
	from $H^{s}(M)$ to $H^{s-m}(M)$, for every $s\in\mathbb R$.
\end{proposition}

Let us denote by $\Sigma$ the class of all matrix symbols, that is,
$$\Sigma \doteq \{ \sigma: \mathbb{N}_0 \ni \ell \mapsto \sigma(\ell) \in \mathbb{C}^{d_\ell \times d_\ell}\}.$$

\begin{definition}\label{moderate-growth}
	We say that a symbol $\sigma \in \Sigma$ has moderate growth if there are $N \in \mathbb{N}$ and $C>0$ such that 
	\begin{equation}\label{moderate-symbol}
	\|\sigma(\ell)\|_{\mathscr{L}(E_{\lambda_\ell})} \leq C(1+\lambda_\ell)^{N/\nu}, \ \ell\in \mathbb{N}_0.
	\end{equation} 
	If $\sigma \in \Sigma$ has moderate growth, the order of $\sigma$ is defined by
	$$ \mbox{ord}(\sigma) \doteq  \inf \{ N \in \mathbb{R}; \ \eqref{moderate-symbol} \mbox{ holds} \}.$$ 
	When the symbol of an $E$-invariant operator $P$ has moderate growth, we define the order of $P$ as being the order of its symbol $\sigma_P$.
\end{definition}

In the remainder of this note, we fix on $M$ a classical positive elliptic pseudo-differential operator $E=E(x,D_x)$ of order $\nu\in \mathbb{R}$. Moreover, whenever we refer to an invariant (or strongly invariant) operator, it shall mean that such invariance occurs with respect to the operator $E$.

%========================================================================================
\section{Global hypoellipticity for strongly invariant operators}\label{sect-GW}
%========================================================================================

Let $P: C^{\infty}(M) \longrightarrow C^{\infty}(M)$ be a strongly invariant operator. By \eqref{invadef2},   for each $\ell\in\mathbb{N}_0$ there exists a matrix $\sigma_P(\ell)\in\mathbb{C}^{d_{\ell}\times d_{\ell}}$ such that
\begin{equation}\label{rel_smth}
\widehat{Pf}(\ell)=\sigma_P(\ell)\widehat{f}(\ell), \ f \in C^{\infty}(M).
\end{equation} 

We claim that the relation \eqref{rel_smth} remains valid for elements of ${\mathscr{D}'}(M)$. Indeed, if $u\in {\mathscr{D}'}(M)$ and $\{u_r\}_{r\in\mathbb{N}}$ is a sequence in $C^{\infty}(M)$ such that $u_r \rightarrow u$ in ${\mathscr{D}'}(M)$, then $\widehat{u}_r(j,k) \rightarrow \widehat{u}(j,k),$ for any $j\in \mathbb{N}_0$ and $1 \leq k \leq d_j$. 

Since $Pu_r \in C^{\infty}(M)$ and $Pu_r \rightarrow Pu$ in ${\mathscr{D}'}(M)$, then for any $j \in\mathbb{N}_0$ and $1 \leq k \leq d_j$ we have $\widehat{Pu_r}(j,k) \rightarrow \widehat{Pu}(j,k)$. 

However, $\widehat{Pu_r}(j,k)= \left(\sigma_P(j)\widehat{u_r}(j)\right)_k,$ therefore  $\widehat{u_r}(j,k) \rightarrow \widehat{u}(j,k)$ and \\ $\widehat{Pu_r}(j,k) \rightarrow \left(\sigma_P(j)\widehat{u}(j)\right)_k.$ This shows that $\widehat{Pu}(j,k) = \left(\sigma_P(j)\widehat{u}(j)\right)_k$ and thus 
\begin{equation}\label{rel_distr} 
\widehat{Pu}(\ell)=\sigma_P(\ell)\widehat{u}(\ell), \mbox{ for all } u\in {\mathscr{D}'}(M).
\end{equation}

\begin{definition}\label{GH-definition}
	An operator $P:{\mathscr{D}'}(M)\to{\mathscr{D}'}(M)$ is globally hypoelliptic on $M$ if the conditions $u \in {\mathscr{D}'}(M)$ and $P u \in C^{\infty}(M)$ imply that $u \in C^{\infty}(M)$.
\end{definition}

To relate the global hypoellipticity of an operator to the behavior of its symbol at infinity, we introduce the following number.

\begin{definition} Let $\sigma \in \Sigma$ be a symbol. For each $\ell \in \mathbb{N}_0$, we define
	\begin{equation*}
	m(\sigma(\ell))  \doteq \inf \{\|\sigma(\ell)v\|; v \in \mathbb{C}^{d_\ell} \text{ and }   \|v\|=1\},
	\end{equation*}
\end{definition}

\begin{theorem} \label{GH} 
	A strongly invariant operator $P:{\mathscr{D}'}(M) \rightarrow {\mathscr{D}'}(M)$ is globally hypoelliptic if and only if there exist constants  $L$, $m$ and $R$ such that
	\begin{equation} \label{LM}
	m(\sigma_P(j)) \geq L(1+\lambda_j)^{m/\nu}, \text{ whenever }  j \geq R.
	\end{equation}
\end{theorem}

\begin{proof} Let $u\in{\mathscr{D}'}(M)$ such that $Pu=f \in C^{\infty}(M)$.  By \eqref{rel_distr} we have 
	$$\widehat{f}(\ell)=\sigma_P(\ell)\widehat{u}(\ell), \ \ell \in \mathbb{N}.$$
	
	By hypothesis, for each $j \geq R$, we have $m(\sigma_P(j)) \neq 0$, that is $\sigma(j)$ is invertible for any $j \geq R$, and we can write 
	$$\widehat{u}(j)=\sigma_P(j)^{-1}\widehat{f}(j).$$
	
	Therefore, if $j \geq R$,
	$$
	\|\widehat{u}(j)\| \leq  \|\sigma_P(j)^{-1}\|\,\|\widehat{f}(j)\| \leq  m(\sigma_P(\ell))^{-1}\|\widehat{f}(j)\|
	\ \leq \ \frac{1}{L}(1+\lambda_j)^{-m/\nu}\|\widehat{f}(j)\|.
	$$
	
	Given $N \in \mathbb{N}$, take $K \in \mathbb{N}$ such that $K > N-m/\nu$. Since $f \in C^{\infty}(M)$, by \eqref{f-smooth}, there is $C_K>0$ such that 
	$$\|\widehat{f}(j)\| \leq C_K(1+\lambda_j)^{-K}, \ j \in \mathbb{N}_0.$$ 
	
	Thus, for $j \geq R$,
	$$
	\|\widehat{u}(j)\| \leq \frac{1}{L} C_K(1+\lambda_j)^{-m/\nu-K} \ \leq \ \frac{C_K}{L}(1+\lambda_j)^{-N}.
	$$
	
	It follows from \ref{f-smooth} that $u \in C^{\infty}(M)$, therefore $P$ is globally hypoelliptic.
	
	On the other hand, proceeding by contradiction, we will construct an element $f \in {\mathscr{D}'}(M) \backslash C^{\infty}(M)$ such that $Pf \in C^{\infty}(M)$, which will prove that $P$ is not globally hypoelliptic, contradicting the hypothesis.
	
	Suppose that for any $L$, $m$, and $R$, it is possible to find $j>R$ such that 
	$$m(\sigma(j)) < L(1+\lambda_j)^{m/\nu}.$$
	
	In particular, for $L=R=1$ and $m=-\nu$, there is $j_1>1$ such that $m(\sigma_P(j_1)) < (1+\lambda_{j_{1}})^{-1}$, thus there exists $a_{j_1} \in \mathbb{C}^{d_{j_1}}$ with $\|a_{j_1}\|=1$ and $\|\sigma_P(j_1)a_{j_1}\| < (1+\lambda_{j_{1}})^{-1}.$
	
	Next, for $L=1$, $R=j_1$ and $m=-2\nu$, there is $j_2>j_1$ such that $m(\sigma_P(j_2)) < (1+\lambda_{j_{2}})^{-2}$, thus there exists $a_{j_2} \in \mathbb{C}^{d_{j_2}}$ with $\|a_{j_2}\|=1$ and $\|\sigma_P(j_2)a_{j_2}\| < (1+\lambda_{j_{2}})^{-2}.$
	
	Proceeding by induction, we obtain a sequence $\{a_{j_k}\}_{k \in\mathbb{N}}$, with $a_{j_k} \in \mathbb{C}^{d_{j_k}}$, $\|a_{j_k}\|=1$ and
	\begin{equation}\label{const}
	\|\sigma_P(j_k)a_{j_k}\| < (1+\lambda_{j_{k}})^{-k},  \mbox{ for all } k \in\mathbb{N},
	\end{equation}
	
	Now define 
	$$
	f \doteq \sum_{\ell=0}^{\infty}\sum_{m=1}^{d_{\ell}}\widehat{f}(\ell,m)e_{\ell}^m,
	$$ 
	where 
	$$
	\widehat{f}(\ell) = 
	\begin{cases}
	a_{j_k}, & \mbox{ if } \ell=j_k \mbox{ for some } k \geq 1, \\
	0, & \mbox{ otherwise.}
	\end{cases}
	$$
	
	Since $\|\widehat{f}(\ell)\| \leq 1\leq (1+\lambda_\ell)$, for all $\ell \in\mathbb{N}_0$, by \eqref{u-distrib} we have $f \in {\mathscr{D}'}(M)$. Moreover, by \eqref{f-smooth} we have $f \notin C^{\infty}(M)$ because $\|\widehat{f}(j_k)\|=1$, for all $k \in \mathbb{N}$.
	
	Now let us prove that $Pf \in C^{\infty}(M)$. Since $P$ is strongly invariant with respect to $E$ we have 
	\begin{equation*}
		Pf \ = \ \sum_{k=0}^{\infty}\sum_{r=1}^{d_{j_k}}\widehat{Pf}(j_k,r)e_{j_k}^r\ = \ \sum_{k=0}^{\infty}\sum_{r=1}^{d_{j_k}}(\sigma_P(j_k)\widehat{f_k}(j_k))_r e_{j_k}^r.
	\end{equation*}
	By \eqref{const} we have 
	$$\|\widehat{Pf}(j_k)\|=\|\sigma_P(j_k)\widehat{f}(j_k)\| \leq (1+\lambda_{j_k})^{-k}.$$
	
	Let $N \in\mathbb{N}$ such that $1+\lambda_{j_k} \geq 1$, for all $k \geq N$. Thus for $k\geq N$ we have
	$$\|\widehat{Pf}(j_k)\| \leq (1+\lambda_{j_k})^{-k} \leq (1+\lambda_{j_k})^{-N},$$ 
	and for $k < N$ we obtain
	$$\|\widehat{Pf}(j_k)\| \leq (1+\lambda_{j_k})^{-k}=(1+\lambda_{j_k})^{N-k}(1+\lambda_{j_k})^{-N}, k \in \mathbb{N}.$$
	Setting $C_N\doteq \max\{(1+\lambda_{j_k})^{N-k}: {1 \leq k \leq N} \}$, then  
	$$\|\widehat{Pf}(j_k)\| \leq C_N(1+\lambda_{j_k})^{-N}, k \in \mathbb{N}.$$
	
	Thus, by condition \ref{f-smooth}, $Pf \in C^{\infty}(M)$, which finishes the proof.
\end{proof}

\begin{definition} \label{exp-hypoel}
	The exponent of hypoellipticity of a globally hypoelliptic operator $P$, denoted $ h(P)$,  is the supreme of all $m\in\mathbb{R}$ such that the condition \eqref{LM} is satisfied. If $P$ is not globally hypoelliptic, we set $ h(P)\doteq -\infty$.
\end{definition}

\begin{remark}
	If $P$ is a globally hypoelliptic invariant operator, then the property \eqref{LM} holds for all $m \leq  h(P)$.
	In particular, if $P$ has order $N$, then $h(P)\leq N.$
\end{remark}

%======================================================================
\section{Compact Lie Groups}\label{Liegroup_section}
%======================================================================
Let $G$ be a compact Lie group and $\mathfrak{g}$ its Lie algebra. By Theorem 3.6.2 of \cite{DK99},  $\mathfrak{g}$ can be written as
$$
\mathfrak{g}=\mathfrak{g}'\oplus \mathfrak{z},
$$
where $\mathfrak{g}'$ is a Lie subalgebra of $\mathfrak{g}$ on which the Killing form is negative definite, and $\mathfrak{z}$ is the kernel of the Killing form. Let $\left\langle{\cdot,\cdot} \right\rangle_{\mathfrak{g'}}$ be the inner product induced by the Killing form and let $\{Y_1, \dots, Y_n\}$ be a orthonormal basis of $\mathfrak{g}'$.  For $\mathfrak{z}$, choose any inner product $\text{Ad}$--invariant and consider $\{ Z_1, \dots Z_m\}$ an orthonormal basis of $\mathfrak{z}$. Observe that the sum of these inner products is an inner product Ad--invariant on $\mathfrak{g}$,
denoted by $\left\langle{\cdot,\cdot}\right\rangle_{\mathfrak{g}}$, 
and we have that $\mathcal{B}=\{Y_1,\dots,Y_n,Z_1,\dots,Z_m \}$ is an orthonormal basis of $\mathfrak{g}$. One can shows that
$$
\mathcal{L}_G= -\sum_{i=1}^n Y_i^2-\sum_{j=1}^mZ_j^2,
$$
is the Laplacian-Beltrami operator on $G$ for the metric induced by $\left\langle{\cdot,\cdot}\right\rangle_\mathfrak{g}$. Notice that
$$
\mathcal{L}_G= \Omega-\sum_{j=1}^mZ_j^2,
$$
where $\Omega$ is the Casimir element of $\mathfrak{g}$, which implies that $\mathcal{L}_G$ commutes with any element of $\mathfrak{g}$. Let $\widehat{G}$ be the set of equivalence classes of irreducible continuous unitary representations of $G$. Since $G$ is compact we have $\widehat{G}$ is a discrete set. Furthermore, for each equivalence class $[\xi_j] \in \widehat{G}$ we may pick a matricial representation $\xi_j:G\to \mathbb{C}^{d_\xi\times d_\xi}$ as representative. We have that the matrix elements of $\xi$ are eigenfunctions of $\mathcal{L}_G$ associated to the same  eigenvalue that we will denote by $-\lambda_{[\xi_j]}^2$, so
$$
\mathcal{L}_G(\xi_j)_{mn} = -\lambda_{[\xi_j]}^2(\xi_j)_{mn}, \quad 1\leq m,n\leq d_{\xi_j}. 
$$
Set 
$$
\{e_j^k\}_{ 1\leq k\leq d_j} = \left\{\sqrt{d_{\xi_j}}(\xi_j)_{mn} \right\}_{1 \leq m,n\leq d_{\xi_j}},
$$
where $d_j:=d_{\xi_j}^2$ and $k$ represents an entry of the matrix $(\xi_j)$ following the lexicographical order:
$$
(m,n) \leq (m',n') \iff m < m' \textrm{ or } \{m=m' \text{ and } n \leq n'\}.
$$
Then we have the subspaces
$$
H_j \equiv H_{[\xi_j]} \doteq \textrm{span} \{ e_j^k; 1 \leq k \leq j\} = \textrm{span} \{(\xi_j)_{mn}; 1 \leq m,n \leq d_{\xi_j} \}.
$$

By Peter-Weyl theorem, we have that $\{e_j^k\}_{1\leq k \leq j}$ is an orthonormal basis of $L^2(G)$ with the norm induced by the normalized Haar measure of $G$.

We point out that the condition \eqref{lambda>0} may not be satisfied because it can occurs $\lambda_{[\xi_j]} = \lambda_{[\xi_{j'}]}$ for $j \neq j'$. Since the eigenspaces of the Laplacian $\mathcal{L}_G$ are finite dimensional, a same eigenvalues can repeat only for finitely many representations and so this is not a problem for the results obtained. 

Let $P:\mathcal{D}'(G) \to \mathcal{D}'(G)$ be a left-invariant operator on $G$. In Section 6 of \cite{DR-Fmul} the authors show that
$$
\sigma_P(j) = \left(
\begin{array}{cccc} 
\tau_P(\xi_j) & 0 & \cdots & 0 \\
0 & \tau_P(\xi_j) & \cdots & 0 \\
\vdots & \vdots & \ddots & \vdots \\
0 & 0 & \cdots & \tau_P(\xi_j)
\end{array}\right) \in \mathbb{C}^{d_j \times d_j}
$$
satisfies the conditions (iii) and (iv) of Proposition \ref{DR_str_inv}, where each element $\tau_P(\xi_j)\in \mathbb{C}^{d_{\xi_j} \times d_{\xi_j}}$ has components $\tau_P(\xi_j)_{mn} = (P\xi_{j_{mn}})(e)$, $1 \leq m,n \leq d_{\xi_j}$, and $e$ is the unit element of $G$. Therefore $P$ is a strongly invariant operator on $G$ with respect to $\mathcal{L}_G$.

Assume that $\tau_P(\xi_j)$ is a diagonalizable matrix, for each $j \in \mathbb{N}_0$. Setting $\lambda_r(\xi_j)$ the eigenvalues of $\tau_P(\xi_j)$, $1 \leq r \leq d_{\xi_j}$, counted with multiplicity, we have that
$$
m(\sigma_P(j)) = m(\tau_P(\xi_j)) = \min_{1 \leq r \leq d_{\xi_j}} |\lambda_r(\xi_j)|.
$$
By Theorem \ref{GH}, the left-invariant operator $P$ is globally hypoelliptic if and only if there exist constants $L$, $N$ and $R$ such that
$$
|\lambda_r(\xi_j)| \geq L\left\langle{\xi_j}\right\rangle^N, \quad \textrm{for all } 1\leq r\leq d_{\xi_j}, \ \textrm{whenever } j \geq R,
$$
where $\left\langle{\xi_j}\right\rangle \doteq (1+\lambda_{[\xi_j]}^2)^{1/2}$.

\begin{example} Let $X\in \mathfrak{g}$, $q\in\mathbb{C}$, and consider the operator $$P=X+q.$$ 
Here $X$ acts on functions as
$$
Xf(x)=\frac{\textrm{d}}{\textrm{d}t} f(x\exp(tX))\bigg|_{t=0}, 
$$
and it extends naturally to distributions as 
$$\left\langle{Xu,f}\right\rangle \doteq -\left\langle{u,Xf}\right\rangle.$$ 

We have that $\tau_X(\xi_j)$ is diagonalizable for every $j\in \mathbb{N}_0$ and its eigenvalues can be written as $i\lambda_r(\xi_j)$, with $\lambda_r(\xi_j) \in \mathbb{R}$, for all $ j \in \mathbb{N}_0$, $1 \leq r \leq d_{\xi_j}$ (see Remark 10.4.20 of \cite{RT10}). 

Thus, $P$ is globally hypoelliptic if and only if there exist constants $L, N$ and $R$ such that
$$
|\lambda_r(\xi_j)-iq| \geq L\left\langle{\xi_j}\right\rangle^N, \ \textrm{for all } 1\leq r\leq d_{\xi_j}, \ \textrm{whenever } j \geq R.
$$ 
In particular, when $q \in\mathbb{R}\setminus\{0\}$, the operator $P=X+q$ is globally hypoelliptic in $G$.
\end{example}

\begin{example}
When $G=\mathbb{S}^3$ we can identify $\widehat{\mathbb{S}^3}$ with $\frac{1}{2}\mathbb{N}_0$ and the symbol of the neutral operator $\partial_0$ can be expressed as 
$$
\tau_{\partial_0}(\ell) = im\delta_{mn},
$$
for all $\ell \in \tfrac{1}{2}\mathbb{N}_0$, $-\ell \leq m,n \leq \ell$, $\ell-m,\ell-n \in \mathbb{N}_0$. Here, the dimension of each eigenspace is $d_\ell=2\ell+1$ and 
$$\left\langle{\ell}\right\rangle  = \sqrt{1+\ell(\ell+1)} \sim 1+\ell.$$ 

Hence, the operator $P=\partial_0+q$ is globally hypoelliptic if and only if there are constants $L,N,R$ such that 
$$
|m-iq| \geq L(1+\ell)^N,
$$
for all $-\ell \leq m \leq \ell, \ \ell-m \in \mathbb{N}_0,$ whenever $\ell \geq R$.

Therefore $P$ is globally hypoelliptic if and only if $q \notin i\frac{1}{2}\mathbb{Z}$, recovering the results from \cite{RTW14}.

\bigskip
Consider now the operator $P=-\mathcal{L}_G + \partial_0^2$. As discussed before, we have that 
$$
\tau_P(\xi_j)_{mn} = (\ell(\ell+1)-m^2)\delta_{mn}.
$$ 
Notice that $\ell^2 - m^2 \geq0$, so
$$
|\ell(\ell+1)-m^2| \geq \ell, \quad \textrm{for all} -\ell \leq m \leq \ell, \ell \in \tfrac{1}{2}\mathbb{N}.
$$
By Theorem 3.3 we conclude that $P$ is globally hypoelliptic with $h(P)=1$. On the other hand, for the operator $P=-\mathcal{L}_G-2\partial_0^2$ we have 
$$
\tau_P(\xi_j)_{mn} = (\ell(\ell+1)-2m^2)\delta_{mn}.
$$ 
Solving the equation $\ell^2+\ell -2m^2=0$ on $\ell$, we obtain 
$$
\ell = \frac{-1+\sqrt{1+8m^2}}{2},
$$
which lead us to the Pell's equation $u^2-8m^2=1$. Notice that $(u_1,m_1)=(3,1)$ is a solution of this equation. Moreover,
$$
\left\{ \begin{array}{cl}
u_{k+1}=&3u_k+8m_k \\
m_{k+1}=&3m_k+u_k
\end{array}
\right.
$$
is also solution of $u^2-8m^2=1$, for all $k \in \mathbb{N}$. We have $\ell \in \mathbb{N}$ because $u_k$ is even, for any $k\in\mathbb{N}$, and we have $m\leq \ell$. Therefore, the operator $P$ is not globally hypoelliptic because its symbol is singular for infinitely many indexes.
\end{example}

\begin{example}
	Let $G=\mathbb{T}^2(\cong \mathbb{R}^2/\mathbb{Z}^2)$ be the two-dimensional torus. Since the eigenfunctions of the Laplacian operator are
	$$
	(t,x)\in \mathbb{T}^2 \mapsto (2\pi)^{-2} e^{2\pi i(\xi t+\eta x)}, \mbox{ with } (\xi,\eta)\in \mathbb{Z}^2,
	$$
	denoting by $H_{(\xi,\eta)}\doteq \emph{\textrm{span}}\{ e^{2\pi i(\xi t+\eta x)}\}$, we have that 
	$$
	L^2(\mathbb{T}^2) = \bigoplus_{\ell\in\mathbb{N}_0} E_\ell, \mbox{ \ where each \ } E_\ell \doteq \bigoplus_{\xi^2+\eta^2=\ell} H_{(\xi,\eta)}.
	$$ 
	
	Finally, from Remark 2.6 of \cite{DR-Fmul}, invariant operators relative to $\{H_{(\xi,\eta)}\}$ are also invariant operators relative to $\{E_\ell\}$.
	
	\medskip
	Consider now the operator 
	$$P=\partial_t+c\partial_x, \mbox{ with } c\in\mathbb{C}$$ 
	
	Clearly $P\mathcal{L}_{\mathbb{T}^2} =  \mathcal{L}_{\mathbb{T}^2}P$ and $P$ is a strongly invariant operator with (matrix) symbol
	$$
	\tau_P(\xi,\eta)=i(\xi+c\eta) \in \mathbb{C}, \ (\xi,\eta)\in \mathbb{Z}^2
	$$
		
	Since $d_{H_{(\xi,\eta)}}=1$ and 
	$$
	\langle(\xi,\eta)\rangle = \sqrt{1+\xi^2+\eta^2} \sim (1+|\xi|+|\eta|).
	$$
	then $P=\partial_t+c\partial_x$ is globally hypoelliptic in $\mathbb{T}^2$ if and only if there are constants $C,N,R$ such that
	\begin{equation}\label{GW}
	|\xi+c\eta| \geq C(1+|\xi|+|\eta|)^N, \ \textrm{whenver } |\xi|+|\eta| \geq R.
	\end{equation}
	
	When \emph{Im}$(c)\neq0$ the condition \eqref{GW} is satisfied because we have $|\xi+c\eta|\geq C$, where $C=\max\{1,\emph{\textrm{Im}}(c) \}$, whenever $(\xi,\eta) \neq (0,0)$. If $c \in \mathbb{Q}$, we obtain infinitely many pairs $(\xi,\eta)\in\mathbb{Z}^2$ such that $|\xi+c\eta|=0$, so there is no $R$ satisfying \eqref{GW}. Finally, for $c\in \mathbb{R}\setminus\mathbb{Q}$ the condition \eqref{GW} is equivalent to say that $c$ is an irrational non-Liouville number. 
	
	Therefore, $P$ is globally hypoelliptic if and only if either $\emph{\textrm{Im}}(c) \neq0$ or $c$ is an irrational non-Liouville number.
\end{example}

%==============================================================================
\section{Global subelliptic estimates}\label{sect-subelliptic}
%==============================================================================

We denote by  $\ker P$ the kernel of a linear operator $P: {\mathscr{D}'}(M)\to{\mathscr{D}'}(M)$, and by $(\ker P)_{H^s}$ the kernel of $P$ in $H^{s}(M)$ which naturally inherits a Hilbert space structure from $H^{s}(M)$.

\begin{lemma} \label{ker}
	Let $P$ be a strongly invariant operator of order $d>0$. If $\ker P \subset C^{\infty}(M)$ then the dimension of $\ker P$ is finite. 
\end{lemma}

\begin{proof}
	By Corollary \ref{sobolev-cont}, $P$ extends to a continuous linear operator from $H^{s}(M)$ to $H^{s-d}(M)$, for every $s\in\mathbb R$.  Let $i: H^{s}(M) \rightarrow H^{s-d}(M)$ be the natural injection, then $i$ maps $(\ker P)_{H^s}$ onto $(\ker P)_{H^{s-d}}$, since $\ker P \subset C^{\infty}(M)$. It follows from the Rellich-Kondrachov Lemma that the inclusion $i:H^s(M)\hookrightarrow H^{s-d}(M)$ is compact, therefore $\ker P = (\ker P)_{H^s}=(\ker P)_{H^{s-d}}$ is finite-dimensional. 
%	{\small \color{red} n\~ao est\'a claro para mim porque a compacidade de $i$ implica em dimens\~ao finita do n\'ucleo.}
\end{proof}

\begin{proposition} \label{sigma*f>Cf}
	Let $P$ be a strongly invariant operator. Then, for all $j \in \mathbb{N}$, there exists $C_j>0$ such that 
	$$\|\sigma_P(j)\widehat{f}(j)\| \geq C_j \|\widehat{f}(j)\|, \mbox{ for all } f \perp (\ker P)_{H^s}.$$ 
\end{proposition}

\begin{proof}
	First, note that if $f \perp (\ker P)_{H^s}$ and $\widehat{f}(j)\neq 0$, for some $j \in \mathbb{N}$, then
	$$\|\sigma_P(j)\widehat{f}(j)\| \neq 0.$$ 
	
	Indeed, suppose that there are $j_0 \in \mathbb{N}$ and $f_0 \perp (\ker P)_{H^s}$ such that $\|\sigma{_P}(j_0) \widehat{f_0}(j_0)\|=0$ and $\widehat{f_0}(j_0)\neq0$. Note that 
	$$f|_{E_{\lambda_{j_0}}}= \sum_{k=1}^{d_{j_0}} \widehat{f}(j_0,k)e_{j_0}^k$$ 
	and, by construction, $P f|_{E_{\lambda_{j_0}}} = (\sigma_P(j_0)\widehat{f_0}(j_0))^{\top}e_{j_0}=0$. 
	
	This way, $f|_{E_{\lambda_{j_0}}}\!\! \in\ker P$ and $\langle f, f|_{E_{\lambda_{j_0}}}\rangle_{H^s}=0$, since $f \perp \ker P$. So $\widehat{f}(j_0)=0$, which leads us to a contradiction.
	
	Now we prove the proposition. Fixed $j  \in \mathbb{N} $, suppose by contradiction that there is a sequence of functions $f_k \perp (\ker P)_{H^s}$ such that $\widehat{f_k}(j)\neq0$ and $$\|\sigma_P(j)\widehat{f_k}(j)\|\leq \frac{1}{k}\|\widehat{f_k}(j)\|, \ k\in\mathbb{N}.$$
	
	Thus, for $h_k=\dfrac{1}{\|\widehat{f_k}(j)\|}\sum\limits_{r=1}^{d_j}\widehat{f_k}(j,r)e_j^r$ we have $h_k \neq 0$ and
	\begin{equation}
	\label{hzero}
	\|\sigma_P(j)\widehat{h_k}(j)\| \leq \frac{1}{k}.
	\end{equation}
	
	Moreover
	$$\|h_k\|_t^2 = \sum_{r=1}^{d_j}(1+\lambda_j)^{2t/\nu}|\widehat{h_k}(j,r)|^2 = d_j(1+\lambda_j)^{2t/\nu}, \ k\in \mathbb{N}_0.$$
	
	Thus, the sequence $\{h_k\}_{k\in\mathbb{N}}$ is limited in $H^t(M)$, for all $t \in \mathbb{R}$. From the Rellich-Kondrachov Lemma, we have that $\{h_k\}$ has, for every $t \in \mathbb{R}$, a convergent subsequence. In particular, by also denoting $\{h_k\}$ the convergent subsequence, there exists $g \in H^s(M)$ such that $h_k \rightarrow g$ in $H^s(M)$, which implies that 
	\begin{equation}\label{s-norm}
	\|g\|_s= \sqrt{d_j}(1+\lambda_j)^{s/\nu}.
	\end{equation}
	
	Since $h_k \perp \ker P$, for each $k \in \mathbb{N}$, we obtain $g \perp \ker P$. By continuity of $P$, we have  $Ph_k \rightarrow Pg$. By \eqref{hzero}, we have $Ph_k \rightarrow 0$. Thus, $Pg=0$ and $g \in \ker P$. Therefore, $g=0$, which contradicts  \eqref{s-norm}.
\end{proof}

For the next result let us recall that the exponent of hypoellipticity $h(P)$ of a globally hypoelliptic operator $P$, is the supreme of all $m\in\mathbb{R}$ such that 
	\begin{equation*}
m(\sigma_P(j)) \geq L(1+\lambda_j)^{m/\nu}, \text{ whenever }  j \geq R,
\end{equation*}
where the constants $L, m$ and $R$ are given by Theorem \ref{GH}.

\begin{proposition} Let $P$ be a strongly invariant operator. If $P$ is globally hypoelliptic, then there is $C>0$, such that, for all $m <h(P)$, we have
	\begin{equation}\label{alpha-thm}
	\|Pf\|_s \geq C\|f\|_{s+m},  \text{ for all } \ f \perp (\ker P)_{H^s}. 
	\end{equation}
\end{proposition}

\begin{proof}
	Since $P$ is globally hypoelliptic, by Theorem \ref{GH} and Definition \ref{exp-hypoel}, for $m<h(P)$, there are $L>0$ and $R\in \mathbb{N}$ such that
	$$m(\sigma_P(j)) \geq L(1+\lambda_j)^{m/\nu},  \mbox{ for all } j \geq R.$$
	And by Theorem \ref{sigma*f>Cf}, for each $j \in \mathbb{N}$, there is $C_j>0$ such that
	$$\|\sigma_P(j)\widehat{f}(j)\| \geq C_j \|\widehat{f}(j)\|, \mbox{ for all } f \perp (\ker P)_{H^s}.$$
	Thus
	\begin{align*}
		\|Pf\|_s^2 & =  \sum_{j=0}^{\infty}(1+\lambda_j)^{2s/\nu}\|\sigma_P(j)\widehat{f}(j)\|^2\\ 
		& \ \geq  \sum_{j=1}^{R-1}(1+\lambda_j)^{2s/\nu} C_j\|\widehat{f}(j)\|^2 + \sum_{j=R}^{\infty} (1+\lambda_j)^{2s/\nu} L^2(1+\lambda_j)^{2m/\nu}\|\widehat{f}(j)\|^2\\ 
		& \ \geq  \widetilde{C}^2\sum_{j=1}^{R-1}(1+\lambda_j)^{2(s+m)/\nu}\|\widehat{f}(j)\|^2 + L^2 \sum_{j=R}^{\infty}(1+\lambda_j)^{2(s+m)/\nu}\|\widehat{f}(j)\|^2 \\ 
		&\geq \ C^2 \sum_{j=1}^{\infty}(1+\lambda_j)^{2(s+m)/\nu}\|\widehat{f}(j)\|^2 = C^2\|f\|_{s+m}^2.
	\end{align*}
	where $\widetilde{C}\doteq \min\{C_j(1+\lambda_j)^{-m/\nu}; 1 \leq j < R\}$ and $C\doteq \min\{\widetilde{C},L\}>0$.
\end{proof}

The last proposition gives a necessary condition for the global hypoellipticity of strongly invariant operators on $M$. On the other hand, it is easy to prove that if inequality \eqref{alpha-thm} holds for any $m> 0$ and $\ker P \subset C^\infty(M)$, then this condition is also sufficient. Therefore, given its importance, we shall highlight this condition for further reference: 
\begin{align}\label{alpha} %\tag{$\alpha$}
\begin{cases}
 \ker P \subset C^\infty(M) \mbox{ and } \exists C>0 \mbox{ such that,  if }  m>0 \mbox{ and } s \in\mathbb{R}  \\
 \mbox{then }  \|Pf\|_s \geq C\|f\|_{s+m}, \mbox{ for all } f \perp (\ker P)_{H^s}.
\end{cases}
\end{align}

\begin{proposition}\label{eqab}
	Let $P$ be a strongly invariant operator of order $d$ and $m>0$. Then $P$ satisfies \eqref{alpha} if and only if there is a constant $K>0$ such that
	\begin{equation} \label{beta}
	\|f\|_{s+m}\leq K(\|f\|_s+\|Pf\|_s),  \ f \in C^{\infty}(M). %\tag{$\beta$}
	\end{equation}
\end{proposition}

\begin{proof} \emph{Sufficiency.} Recall that $P$ extends to a continuous linear operator on all Sobolev spaces. Therefore, if $f \in C^{\infty}(M)$, we can write $f=f_1+f_2$, with $f_1 \in(\ker P)_{H^s}$ and $f_2\perp(\ker P)_{H^s}$. Thus, $Pf=Pf_2$ and $\|f\|_s^2 = \|f_1\|_s^2 + \|f_2\|_s^2$. In particular, $\|f\|_s \geq \|f_1\|_s$.
	
	Since $\ker P \subset C^{\infty}(M)$, by Lemma \ref{ker}, the dimension of $\ker P$ is finite and all the norms on $\ker P$ are equivalent. Therefore, there is $K_1>0$ such that 
	$$\|g\|_{s+m} \leq K_1\|g\|_s,  \  g\in \ker P.$$
	
	By \eqref{alpha}, we have $\|Pf_2\|_s \geq C\|f_2\|_{s+m}$, thus
	\begin{align*} 
		\|f\|_{s+m} & \leq \ \|f_1\|_{s+m}+\|f_2\|_{s+m} \ \leq \
		K_1\|f_1\|_s+C^{-1}\|Pf_2\|_s\\ 
		& \leq \  K_1\|f\|_s+C^{-1}\|Pf\|_s\ \leq \ K(\|f\|_s+\|Pf\|_s).
	\end{align*}

	\noindent \emph{Necessity.} Let $f \in {\mathscr{D}'}(M)$ such that $Pf=0$. Since ${\mathscr{D}'}(M)=\bigcup_s H^s(M)$, then we have $f \in H^s(M)$ for some $s \in \mathbb{R}$. By \eqref{beta} we have $f \in H^{s+m}(M)$ and replacing $s$ by $s+m$ we get $f \in H^{s+2m}$. By induction we have $f \in \bigcap_s{H^s(M)}=C^{\infty}(M),$ hence $\ker P \subset C^{\infty}(M)$.
	
	Now, assume that the inequality \eqref{alpha} is not valid, then it is possible to obtain a sequence of functions $f_j \perp (\ker P)_{H^s}$ such that $\|f_j\|_{s+m}=1$, for all $j \in \mathbb{N}$ and $\|Pf_j\|_s \rightarrow 0$, as $j\rightarrow \infty$. 
	
	By the Rellich-Kondrachov Lemma, $\{f_{j}\}$ has a convergent subsequence \\ $f_{j_k} \rightarrow g$ in $H^s(M)$ and, by continuity, we have $Pf_{j_k} \rightarrow Pg$ in $H^{s-d}(M)$. Since $\|Pf_j\|_s\rightarrow 0$, we have $\|Pf_j\|_{s-d} \rightarrow 0$, therefore $Pg=0$ and $g \in \ker P$. However, $f_j \perp (\ker P)_{H^s}$ and $f_j \rightarrow g \in H^s(M)$, hence $g \perp (\ker P)_{H^s}$. 
	
	In this way, we have  $g \in (\ker P)_{H^s} \cap (\ker P)^{\perp}_s$, which implies that  $g=0$. On the other hand, by \eqref{beta}, 
	$1=\|f_j\|_{s+m} \leq K(\|f_j\|_s+\|Pf_j\|_s).$
	When $j \rightarrow \infty$ we have $1 \leq K\|g\|_s = 0$, which is a contradiction. So the inequality \eqref{alpha} is true.
\end{proof}

\begin{proposition}
	Let $P$ be a strongly invariant operator of order $d\geq 0$ and $m>0$. Then $\eqref{beta}$ implies that 
	\begin{align}\label{gamma} %\tag{$\gamma$}
	\begin{cases}
	\ker P \subset C^\infty(M) \mbox{ and  } P(C^{\infty}(M)) \mbox{ is closed }\\
	\mbox{in } C^{\infty}(M) \mbox{ with the } {\mathscr{D}'}(M) \mbox{ relative topology,}
	\end{cases}
	\end{align}
	that is, if $f_j, g \in C^{\infty}(M)$ and $Pf_j \rightarrow g$ in $H^s(M)$, for some $s \in\mathbb{R}$, then $g=Ph$, for some $h\in C^{\infty}(M)$.
\end{proposition}

\begin{proof} 
	By using the same arguments from the proof of Proposition \ref{eqab}, we have $\ker P \subset C^{\infty}(M)$. So, let us show that $P(C^{\infty}(M))$ is closed in $C^{\infty}(M)$ with the ${\mathscr{D}'}(M)$ relative topology.
	
	Let $f_j, g \in C^{\infty}(M)$ such that $Pf_j \rightarrow g$ in $H^s(M)$, then we can assume that $f_j \perp (\ker P)_{H^s}$, for all $j\in\mathbb{N}$. Indeed, for each $f_j \in C^{\infty}(M)$ we can write $f_j=f_{1j}+f_{2j}$, with $f_{1j} \in (\ker P)_{H^s}$ and $f_{2j} \perp (\ker P)_{H^s}$. Since $\ker P \subset C^{\infty}(M)$ and $f_j \in C^{\infty}(M)$, we have $f_{2j} \in C^{\infty}(M)$ and $Pf_{2j} = Pf_j \rightarrow g$.
	
	Let us treat the cases when $\{\|f_j\|_s\}$ is bounded and when $\{\|f_j\|_s\}$ is unbounded separately. 
	
	First assume that $\{\|f_j\|_s\}$ is bounded. Since $\{Pf_j\}$ is convergent in $H^s(M)$, the sequence $\{\|Pf_j\|_s\}$ is bounded and, by \eqref{beta}, we have that $\{\|f_j\|_{s+m}\}$ is bounded. Thus, by the Rellich-Kondrachov Lemma, the sequence $\{f_j\}$ has a convergent subsequence in $H^s(M)$, which we continue to denote $\{f_j\}$. Let $h \in H^s(M)$ such that $f_j \rightarrow h$ in $H^s(M)$, by continuity of $P$, we have $Pf_j \rightarrow Ph$ in $H^{s-d}(M)$. Since $Pf_j \rightarrow g$ in $H^s(M)$ and $d\geq 0$, then $s-d<s$ and we have $Ph=g$. 
	
	Finally, by \eqref{beta}, $\|h\|_{s+m} \leq K(\|h\|_s+\|Ph\|_s)= K(\|h\|_s+\|g\|_s).$ Thus,   
	$h \in H^{s+m}(M)$. By induction we have $h \in \bigcap_s H^s(M) = C^{\infty}(M)$.
	
	Now, assume that $\{\|f_j\|_s\}$ is unbounded. Then it is possible to obtain a subsequence, which we continue to denote $\{f_j\}$, such that $\|f_j\|_s \rightarrow \infty$. 
	
	Since $\{\|Pf_j\|_s\}$ is bounded, because $Pf_j \rightarrow g$ in $H^s(M)$, setting $\widetilde{f}_j = f_j/\|f_j\|_s$ we have 
	$$\|P\widetilde{f}_j\|_s=\dfrac{\|Pf_j\|_s}{\|f_j\|_s} \longrightarrow 0,$$
	
	By \ref{beta}, $\|\widetilde{f}_j\|_{s+m} \leq K(\|\widetilde{f}_j\|_s+\|P\widetilde{f}_j\|_s),$ which implies that $\{\|\widetilde{f}_j\|_{s+m}\}$ is bounded. Now, by the Rellich-Kondrachov Lemma, this sequence has a convergent subsequence in $H^s(M)$, which we continue to denote by $\{\widetilde{f}_j\}$. Thus, $\widetilde{f}_j \rightarrow t \in H^s(M)$ and $Pt=0$, hence $t \in \ker P$. 
	
	However, $\widetilde{f}_j \perp \ker P$ in $H^s(M)$, thus $t \perp \ker P$ and $t=0$. Moreover
	$$\|t\|_s = \lim_{j \rightarrow \infty}{\|\widetilde{f}_j\|_s}=1,$$
	which contradicts the statement of $t=0$. Then $\{\|f_j\|_s\}$ must be bounded, once we take it as perpendicular to $\ker P$.
\end{proof}

\begin{theorem}
	\label{gammaGH}
	Any strongly invariant operator $P$, defined on $M$, satisfying condition \eqref{gamma} is globally hypoelliptic.
\end{theorem}

\begin{proof}
	Let $P$ be a strongly invariant operator on $M$ and assume that $Pf=g \in C^{\infty}(M)$, with $f \in {\mathscr{D}'}(M)$. Since ${\mathscr{D}'}(M)=\bigcup_s H^s(M)$, then $f \in H^s(M)$, for some $s\in\mathbb{R}$. By density, we obtain a sequence $\{f_j\}_j$ in $C^{\infty}(M)$ such that $f_j \rightarrow f$ in $H^s(M)$, and therefore $Pf_j \rightarrow Pf=g$ in $H^{s-d}(M)$. Thus, by \eqref{gamma}, there is $h \in C^{\infty}(M)$ such that $Ph=g$ and $P(f-h)=Pf-Ph=g-g=0,$ that is, $f-h \in \ker P$. 
	
	Since $\ker P \subset C^{\infty}(M)$, we have $f-h \in C^{\infty}(M)$. Thus, $f=(f-h)+h \in C^{\infty}(M)$ and $P$ is globally hypoelliptic.
\end{proof}

\section*{Acknowledgments}
This study was financed in part by the Coordenação de Aperfeiçoamento de Pessoal de Nível Superior - Brasil (CAPES) - Finance Code 001.

%\addcontentsline{toc}{section}{References}
\section*{References}
\biboptions{sort&compress}
\bibliographystyle{model5-names}
\bibliography{references}

\end{document}